\documentclass[12pt,reqno]{amsart}
\usepackage{eurosym}
\usepackage{amsfonts}
\usepackage{amssymb,amsmath}
\usepackage{amscd}
\usepackage{hyperref}
\usepackage{graphicx}
\usepackage{epsfig}
\usepackage{epstopdf}
\usepackage{caption}
\usepackage{subcaption}

\newtheorem{theorem}{Theorem}[section]

\newtheorem{proposition}{Proposition}[section]
\newtheorem{corollary}{Corollary}[section]
\newtheorem{definition}{Definition}[section]
\newtheorem{example}{Example}[section]
\newtheorem{remark}{Remark}[section]
\numberwithin{equation}{section}
\begin{document}
\title[Some fixed-circle theorems on metric spaces]{\textbf{Some
fixed-circle theorems on metric spaces}}
\author[N. YILMAZ \"{O}ZG\"{U}R ]{N\.{I}HAL YILMAZ \"{O}ZG\"{U}R$^{1}$}
\address{Bal\i kesir University\\
Department of Mathematics\\
10145 Bal\i kesir, TURKEY}
\email{nihal@balikesir.edu.tr}
\author[N. TA\c{S}]{N\.{I}HAL TA\c{S}$^{1}$}
\address{Bal\i kesir University\\
Department of Mathematics\\
10145 Bal\i kesir, TURKEY}
\email{nihaltas@balikesir.edu.tr}
\thanks{$^{1}$Bal\i kesir University, Department of Mathematics, 10145 Bal\i
kesir, TURKEY}

\begin{abstract}
The fixed-point theory and its applications to various areas of science are
well known. In this paper we present some existence and uniqueness theorems
for fixed circles of self-mappings on metric spaces with geometric
interpretation. We verify our results by illustrative examples.
\end{abstract}

\maketitle

\section{\textbf{Introduction}}

\label{intro} It has been extensively studied the existence of fixed points
of functions which satisfy certain conditions since the time of Stefan
Banach. At first we recall the Banach contraction principle as follows:

\begin{theorem}
\cite{Ciesielski-2007} \label{thm3} Let $(X,d)$ be a complete metric space
and a self-mapping $T:X\rightarrow X$ be a contraction, that is, there
exists some $h\in \lbrack 0,1)$ such that%
\begin{equation*}
d(Tx,Ty)\leq hd(x,y)\text{,}
\end{equation*}%
for any $x,y\in X$. Then there exists a unique fixed point $x_{0}\in X$ of $%
T $.
\end{theorem}

Since then many authors have been studied new contractive conditions for
fixed-point theorems. For example, Caristi gave the following fixed-point
theorem.

\begin{theorem}
\cite{Caristi-1976} \label{thm10} Let $(X,d)$ be a complete metric space and
$T:X\rightarrow X$. If there exists a lower semicontinuous function $\varphi
$ mapping $X$ into the nonnegative real numbers%
\begin{equation}
d(x,Tx)\leq \varphi (x)-\varphi (Tx)\text{,}  \label{caristi_eqn}
\end{equation}%
$x\in X$ then $T$ has a fixed point.
\end{theorem}

In \cite{Rhoades}, Rhoades defined the following condition (which is called
Rhoades' condition):%
\begin{equation*}
d(Tx,Ty)<\max \left\{ d(x,y),d(x,Tx),d(y,Ty),d(x,Ty),d(y,Tx)\right\} \text{,}
\end{equation*}%
for all $x,y\in X$ with $x\neq y$.

In some special metric spaces, mappings with fixed points have been used in
neural networks as activation functions. For example, M\"{o}bius
transformations have been used for this purpose. It is known that a M\"{o}%
bius transformation is a rational function of the form%
\begin{equation}
T(z)=\frac{az+b}{cz+d},  \label{Mobius}
\end{equation}%
where $a$, $b$, $c$, $d$ are complex numbers satisfying $ad-bc\neq 0$. A M%
\"{o}bius transformation has at most two fixed points (see \cite%
{Jones-Singerman} for more details about M\"{o}bius transformations). In
\cite{Mandic-2000}, Mandic identified the activation function of a neuron
and a single-pole all-pass digital filter section as M\"{o}bius
transformations. He observed that the fixed points of a neural network were
determined by the fixed points of the employed activation function. So the
existence of the fixed points of an activation function were guaranteed by
the underlying M\"{o}bius transformation (one or two fixed points).

On the other hand, there are some examples of functions which fix a circle.
For example, let $%
\mathbb{C}
$ be the metric space with the usual metric%
\begin{equation*}
d(z,w)=\left\vert z-w\right\vert \text{,}
\end{equation*}%
for all $z,w\in
\mathbb{C}
$. Let the mapping $T$ be defined as%
\begin{equation*}
Tz=\dfrac{1}{\overline{z}}\text{,}
\end{equation*}%
for all $z\in
\mathbb{C}
\setminus \left\{ 0\right\} $. The mapping $T$ fixes the unit circle $%
C_{0,1} $. In \cite{Ozdemir-2011}, \"{O}zdemir, \.{I}skender and \"{O}zg\"{u}%
r used new types of activation functions which fix a circle for a complex
valued neural network (CVNN). The usage of these types activation functions
lead us to guarantee the existence of fixed points of the complex valued
Hopfield neural network (CVHNN).

Therefore it is important the notions of \textquotedblleft fixed
circle\textquotedblright\ and \textquotedblleft mappings with a fixed
circle\textquotedblright . It will be an interesting problem to study some
fixed-circle theorems on general spaces (metric spaces or normed spaces).

Motivated by the above studies, our aim in this paper is to examine some
fixed-circle theorems for self-mappings on metric spaces. Also we determine
the uniqueness conditions of these theorems. In Section \ref{sec:1} we
introduce the notion of a fixed circle and prove three theorems for the
existence of fixed circles of self-mappings on metric spaces. Also we give
some necessary examples for obtained fixed-circle theorems. In Section \ref%
{sec:2} we present some self-mappings which have at least two fixed circles.
Hence we give three uniqueness theorems for the fixed-circle theorems
obtained in Section \ref{sec:1}.

\section{\textbf{Existence of the self-mappings with fixed circles}}

\label{sec:1} In this section we give fixed-circle theorems under some
conditions on metric spaces and obtain some examples of mappings which have
or not fixed circles. At first we give the following definition.

\begin{definition}
\label{def1} Let $(X,d)$ be a metric space and $C_{x_{0},r}=\{x\in
X:d(x_{0},x)=r\}$ be a circle. For a self-mapping $T:X\rightarrow X$, if $%
Tx=x$ for every $x\in C_{x_{0},r}$ then we call the circle $C_{x_{0},r}$ as
the fixed circle of $T$.
\end{definition}

Now we give the following existence theorem for a fixed circle using the
inequality (\ref{caristi_eqn}).

\begin{theorem}
\label{thm1} Let $(X,d)$ be a metric space and $C_{x_{0},r}$ be any circle
on $X$. Let us define the mapping
\begin{equation}
\varphi :X\rightarrow \lbrack 0,\infty ),\varphi (x)=d(x,x_{0}),
\label{phi function}
\end{equation}%
for all $x\in X$. If there exists a self-mapping $T:X\rightarrow X$
satisfying

$(C1)$ $d(x,Tx)\leq \varphi (x)-\varphi (Tx)$\newline
and

$(C2)$ $d(Tx,x_{0})\geq r$,\newline
for each $x\in C_{x_{0},r}$, then the circle $C_{x_{0},r}$ is a fixed circle
of $T$.
\end{theorem}

\begin{proof}
Let us consider the mapping $\varphi $ defined in (\ref{phi function}). Let $%
x\in C_{x_{0},r}$ be any arbitrary point. We show that $Tx=x$ whenever $x\in
C_{x_{0},r}$. Using the condition $(C1)$ we obtain%
\begin{eqnarray}
d(x,Tx) &\leq &\varphi (x)-\varphi (Tx)=d(x,x_{0})-d(Tx,x_{0})  \label{eqn1}
\\
&=&r-d(Tx,x_{0})\text{.}  \notag
\end{eqnarray}%
Because of the condition $(C2)$, the point $Tx$ should be lies on or
exterior of the circle $C_{x_{0},r}$. Then we have two cases. If $%
d(Tx,x_{0})>r$ then using $($\ref{eqn1}$)$ we have a contradiction.
Therefore it should be $d(Tx,x_{0})=r$. In this case, using $($\ref{eqn1}$)$
we get%
\begin{equation*}
d(x,Tx)\leq r-d(Tx,x_{0})=r-r=0
\end{equation*}%
and so $Tx=x$.

Hence we obtain $Tx=x$ for all $x\in C_{x_{0},r}$. Consequently, the
self-mapping $T$ fixes the circle $C_{x_{0},r}$.
\end{proof}

\begin{remark}
\label{rem3} $1)$ We note that Theorem \ref{thm10} guarantees the existence
of a fixed point while Theorem \ref{thm1} guarantees the existence of a
fixed circle. In the cases where the circle $C_{x_{0},r}$ has only one
element $($see Example \ref{exm13} for an example$)$ Theorem \ref{thm1} is a
special case of Theorem \ref{thm10}.

$2)$ Notice that the condition $(C1)$ guarantees that $Tx$ is not in the
exterior of the circle $C_{x_{0},r}$ for each $x\in C_{x_{0},r}$. Similarly
the condition $(C2)$ guarantees that $Tx$ is not in the interior of the
circle $C_{x_{0},r}$ for each $x\in C_{x_{0},r}$. Consequently, $Tx\in
C_{x_{0},r}$ for each $x\in C_{x_{0},r}$ and so we have $T(C_{x_{0},r})%
\subset C_{x_{0},r}$ $($see Figure \ref{fig:1} for the geometric
interpretation of the conditions $(C1)$ and $(C2))$.

\begin{figure}[t]
\centering
\begin{subfigure}{.5\textwidth}
  \centering
  \includegraphics[width=.4\linewidth]{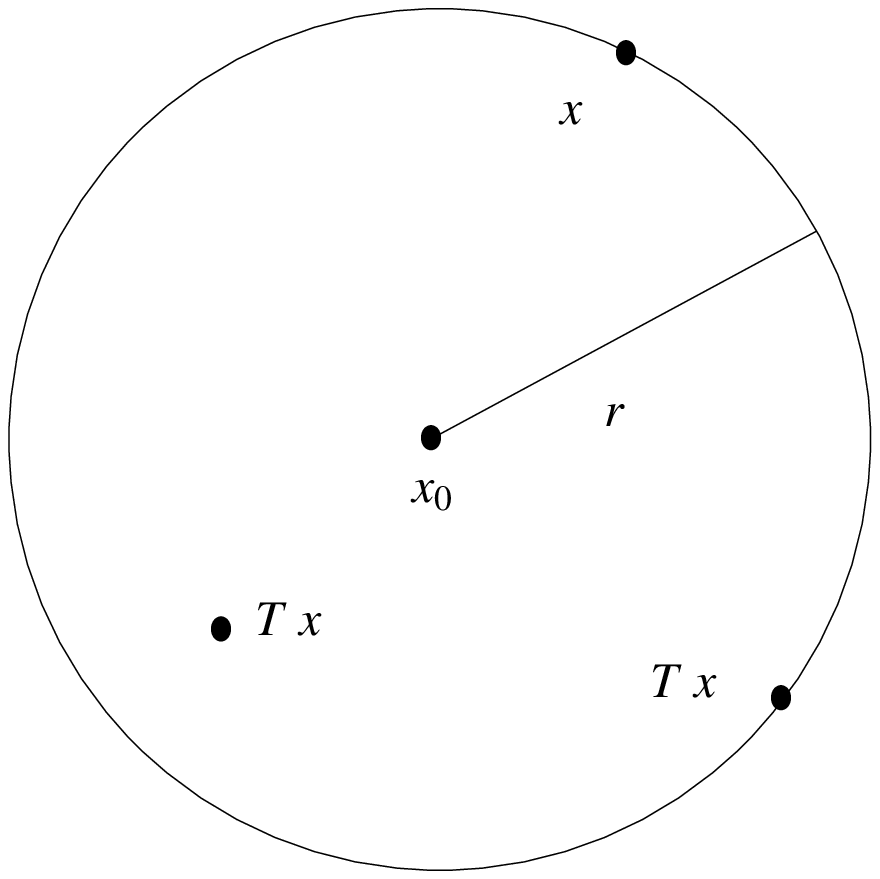}
  \caption{\Small The condition $(C1)$.}
  \label{fig:1A}
\end{subfigure}
\begin{subfigure}{.5\textwidth}
  \centering
  \includegraphics[width=.4\linewidth]{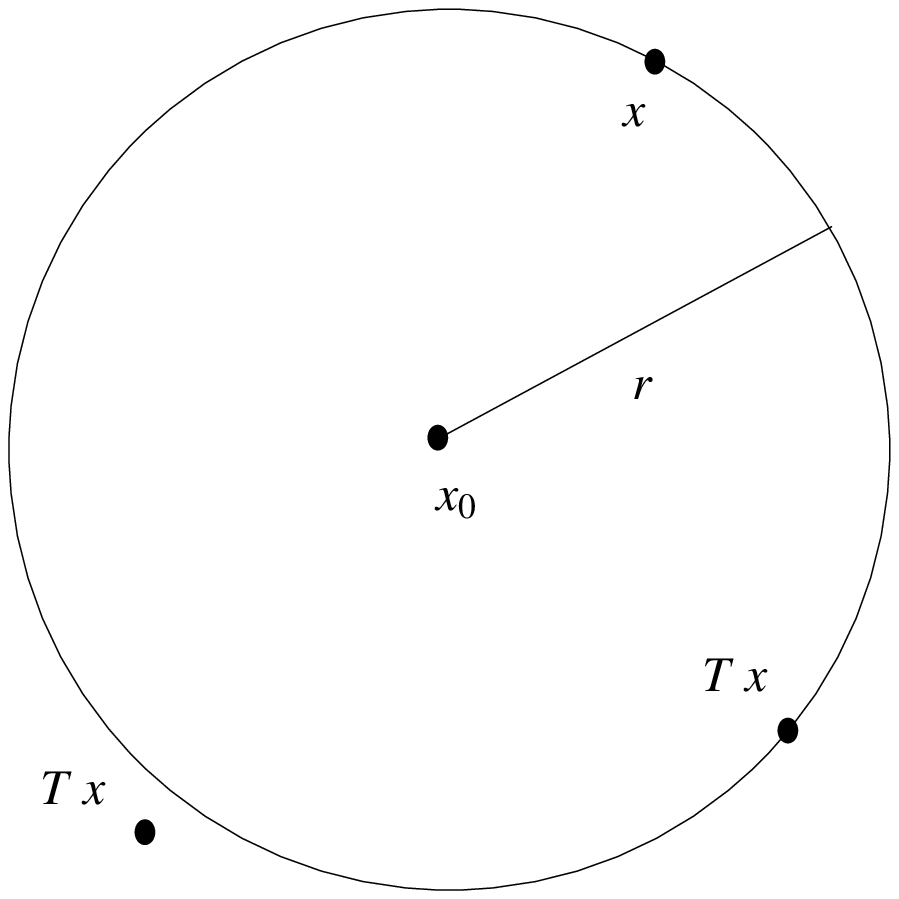}
  \caption{\Small  The condition $(C2)$.}
  \label{fig:1B}
\end{subfigure}
\begin{subfigure}{.5\textwidth}
  \centering
  \includegraphics[width=.4\linewidth]{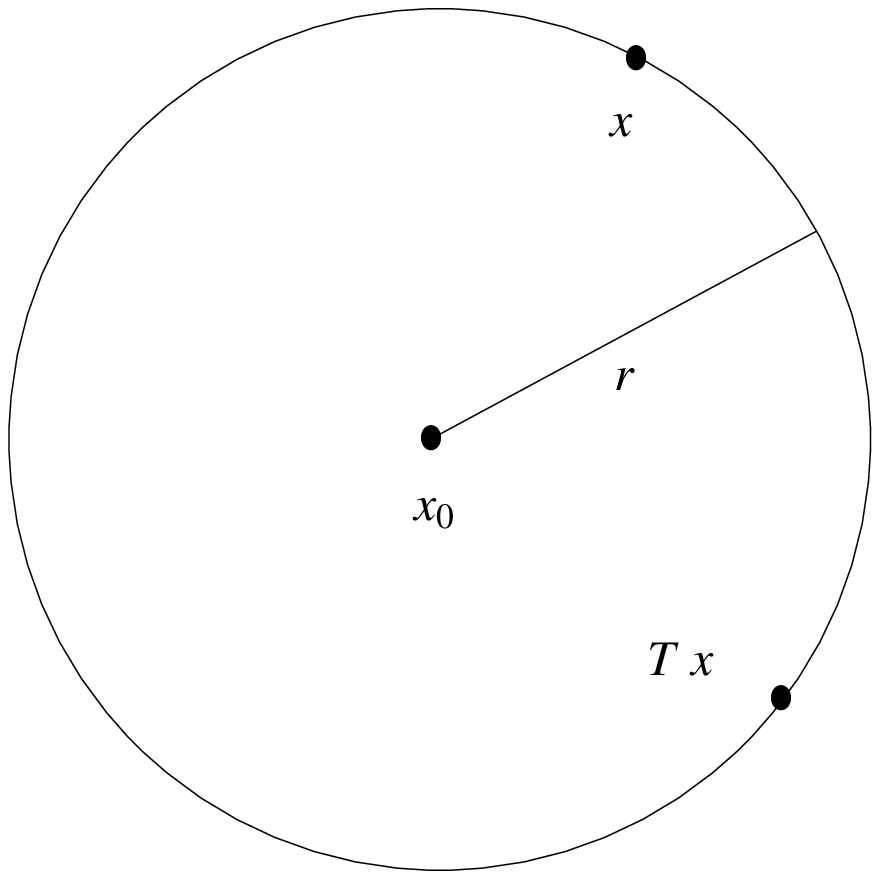}
  \caption{\Small  The condition $(C1) \cap (C2)$.}
  \label{fig:1B}
\end{subfigure}
\caption{\Small The geometric interpretation of the conditions $(C1)$ and $%
(C2) $.}
\label{fig:1}
\end{figure}
\end{remark}

Now we give a fixed-circle example.

\begin{example}
\label{exm1} Let $(X,d)$ be a metric space and $\alpha $ be a constant such
that%
\begin{equation*}
d(\alpha ,x_{0})>r\text{.}
\end{equation*}%
Let us consider a circle $C_{x_{0},r}$ and define the self-mapping $%
T:X\rightarrow X$ as%
\begin{equation*}
Tx=\left\{
\begin{array}{ccc}
x & \text{;} & x\in C_{x_{0},r} \\
\alpha & \text{;} & \text{otherwise}%
\end{array}%
\right. \text{,}
\end{equation*}%
for all $x\in X$. Then it can be easily seen that the conditions $(C1)$ and $%
(C2)$ are satisfied. Clearly $C_{x_{0},r}$ is a fixed circle of $T$.
\end{example}

Now, in the following examples, we give some examples of self-mappings which
satisfy the condition $(C1)$ and do not satisfy the condition $(C2)$.

\begin{example}
\label{exm2} Let $(X,d)$ be any metric space, $C_{x_{0},r}$ be any circle on
$X$ and the self-mapping $T:X\rightarrow X$ be defined as%
\begin{equation*}
Tx=x_{0}\text{,}
\end{equation*}%
for all $x\in X$. Then the self-mapping $T$ satisfies the condition $(C1)$
but does not satisfy the condition $(C2)$. Clearly $T$ does not fix the
circle $C_{x_{0},r}$.
\end{example}

\begin{example}
\label{exm3} Let $(%
\mathbb{R}
,d)$ be the usual metric space. Let us consider the circle $C_{1,2}$ and
define the self-mapping $T:%
\mathbb{R}
\rightarrow
\mathbb{R}
$ as%
\begin{equation*}
Tx=\left\{
\begin{array}{ccc}
1 & \text{;} & x\in C_{1,2} \\
2 & \text{;} & \text{otherwise}%
\end{array}%
\right. \text{,}
\end{equation*}%
for all $x\in
\mathbb{R}
$. Then the self-mapping $T$ satisfies the condition $(C1)$ but does not
satisfy the condition $(C2)$. Clearly $T$ does not fix the circle $C_{1,2}$ $%
($or any circle$)$.
\end{example}

In the following examples, we give some examples of self-mappings which
satisfy the condition $(C2)$ and do not satisfy the condition $(C1)$.

\begin{example}
\label{exm4} Let $(X,d)$ be any metric space and $C_{x_{0},r}$ be any circle
on $X$. Let $\alpha $ be chosen such that $d(\alpha ,x_{0})=\rho >r$ and
consider the self-mapping $T:X\rightarrow X$ defined by%
\begin{equation*}
Tx=\alpha \text{,}
\end{equation*}%
for all $x\in X$. Then the self-mapping $T$ satisfies the condition $(C2)$
but does not satisfy the condition $(C1)$. Clearly $T$ does not fix the
circle $C_{x_{0},r}$.
\end{example}

\begin{example}
\label{exm5} Let $(%
\mathbb{C}
,d)$ be the usual complex metric space and $C_{0,1}$ be the unit circle on $%
\mathbb{C}
$. Let us consider the self-mapping $T:%
\mathbb{C}
\rightarrow
\mathbb{C}
$ defined by%
\begin{equation*}
Tz=\left\{
\begin{array}{ccc}
\dfrac{1}{z} & ; & z\neq 0 \\
0 & ; & z=0%
\end{array}%
\right. \text{,}
\end{equation*}%
for all $z\in
\mathbb{C}
$. Then the self-mapping $T$ satisfies the condition $(C2)$ but does not
satisfy the condition $(C1)$. Clearly $T$ does not fix the circle $C_{0,1}$ $%
($or any circle$)$. Notice that $T$ fixes only the points $-1$ and $1$ on
the unit circle.
\end{example}

Now we give another existence theorem for fixed circles.

\begin{theorem}
\label{thm5} Let $(X,d)$ be a metric space and $C_{x_{0},r}$ be any circle
on $X$. Let the mapping $\varphi $ be defined as $($\ref{phi function}$)$.
If there exists a self-mapping $T:X\rightarrow X$ satisfying

$(C1)^{\ast }$ $d(x,Tx)\leq \varphi (x)+\varphi (Tx)-2r$\newline
and

$(C2)^{\ast }$ $d(Tx,x_{0})\leq r$,\newline
for each $x\in C_{x_{0},r}$, then $C_{x_{0},r}$ is a fixed circle of $T$.
\end{theorem}

\begin{proof}
We consider the mapping $\varphi $ defined in $($\ref{phi function}$)$. Let $%
x\in C_{x_{0},r}$ be any arbitrary point. Using the condition $(C1)^{\ast }$
we obtain%
\begin{eqnarray}
d(x,Tx) &\leq &\varphi (x)+\varphi (Tx)-2r=d(x,x_{0})+d(Tx,x_{0})-2r
\label{eqn2} \\
&=&d(Tx,x_{0})-r\text{.}  \notag
\end{eqnarray}%
Because of the condition $(C2)^{\ast }$, the point $Tx$ should be lies on or
interior of the circle $C_{x_{0},r}$. Then we have two cases. If $%
d(Tx,x_{0})<r$ then using $($\ref{eqn2}$)$ we have a contradiction.
Therefore it should be $d(Tx,x_{0})=r$. If $d(Tx,x_{0})=r$ then using $($\ref%
{eqn2}$)$ we get%
\begin{equation*}
d(x,Tx)\leq d(Tx,x_{0})-r=r-r=0
\end{equation*}%
and so we find $Tx=x$.

Consequently, $C_{x_{0},r}$ is a fixed circle of $T$.
\end{proof}

\begin{remark}
\label{rem4} Notice that the condition $(C1)^{\ast }$ guarantees that $Tx$
is not in the interior of the circle $C_{x_{0},r}$ for each $x\in
C_{x_{0},r} $. Similarly the condition $(C2)^{\ast }$ guarantees that $Tx$
is not in the exterior of the circle $C_{x_{0},r}$ for each $x\in
C_{x_{0},r} $. Consequently, $Tx\in C_{x_{0},r}$ for each $x\in C_{x_{0},r}$
and so we have $T(C_{x_{0},r})\subset C_{x_{0},r}$ $($see Figure \ref{fig:2}
for the geometric interpretation of the conditions $(C1)^{\ast }$ and $%
(C2)^{\ast })$.

\begin{figure}[t]
\centering
\begin{subfigure}{.5\textwidth}
  \centering
  \includegraphics[width=.4\linewidth]{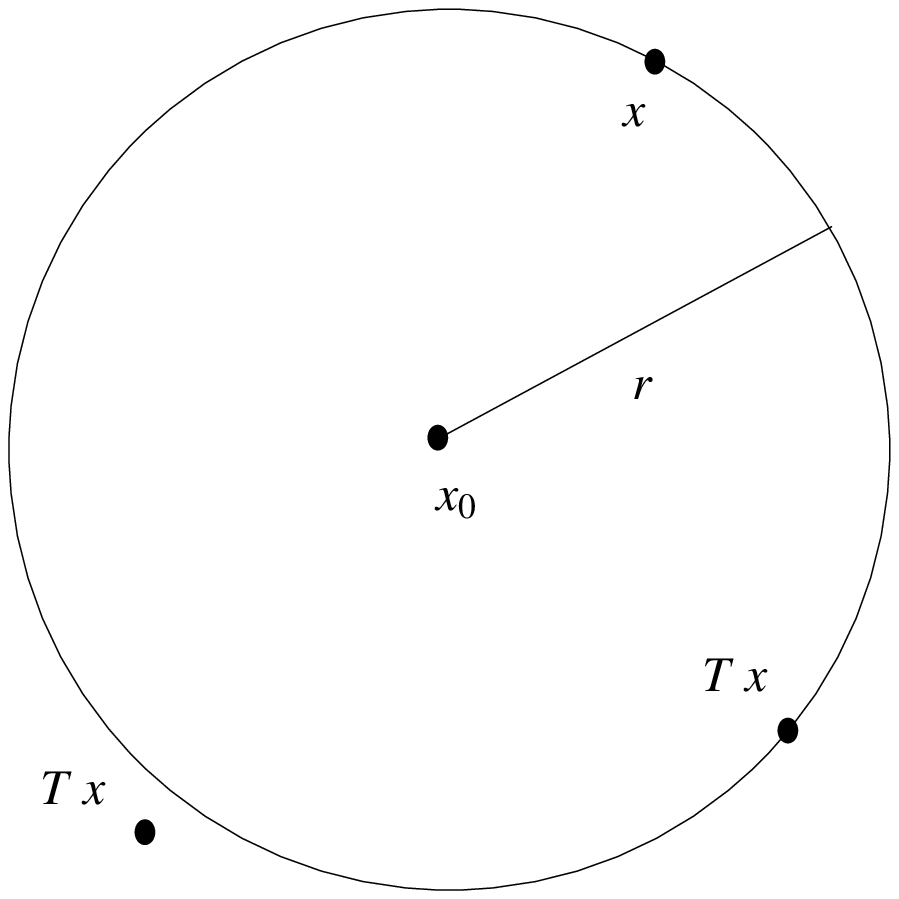}
  \caption{\Small The condition $(C1)^{*}$.}
  \label{fig:1A}
\end{subfigure}
\begin{subfigure}{.5\textwidth}
  \centering
  \includegraphics[width=.4\linewidth]{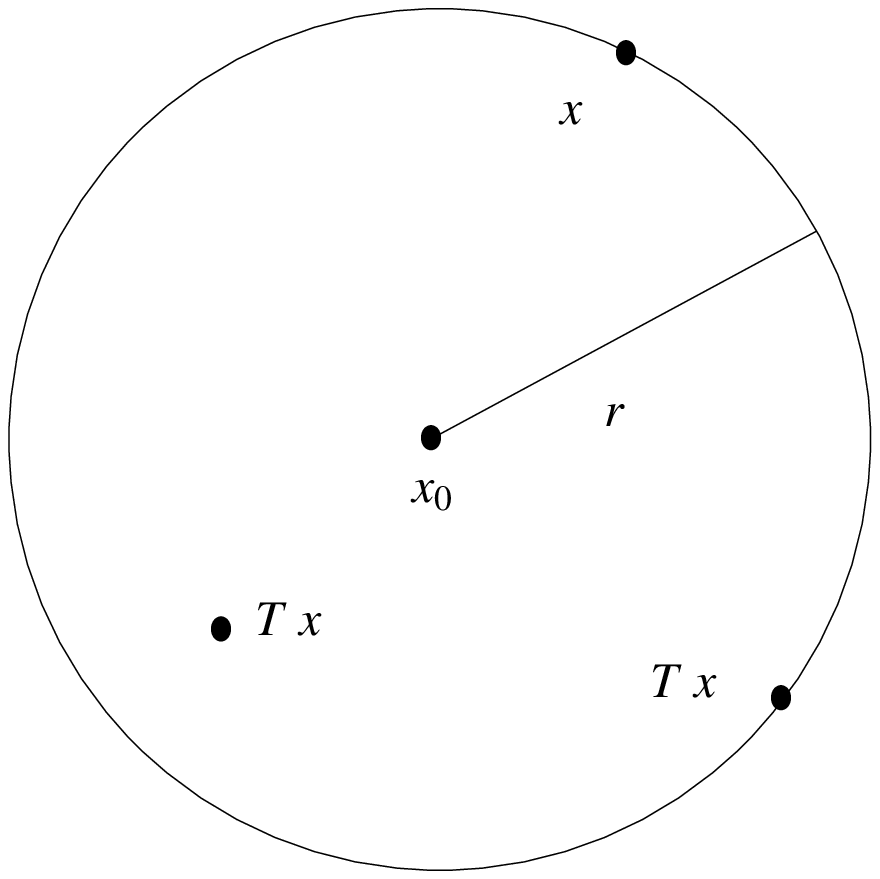}
  \caption{\Small  The condition $(C2)^{*}$.}
  \label{fig:1B}
\end{subfigure}
\begin{subfigure}{.5\textwidth}
  \centering
  \includegraphics[width=.4\linewidth]{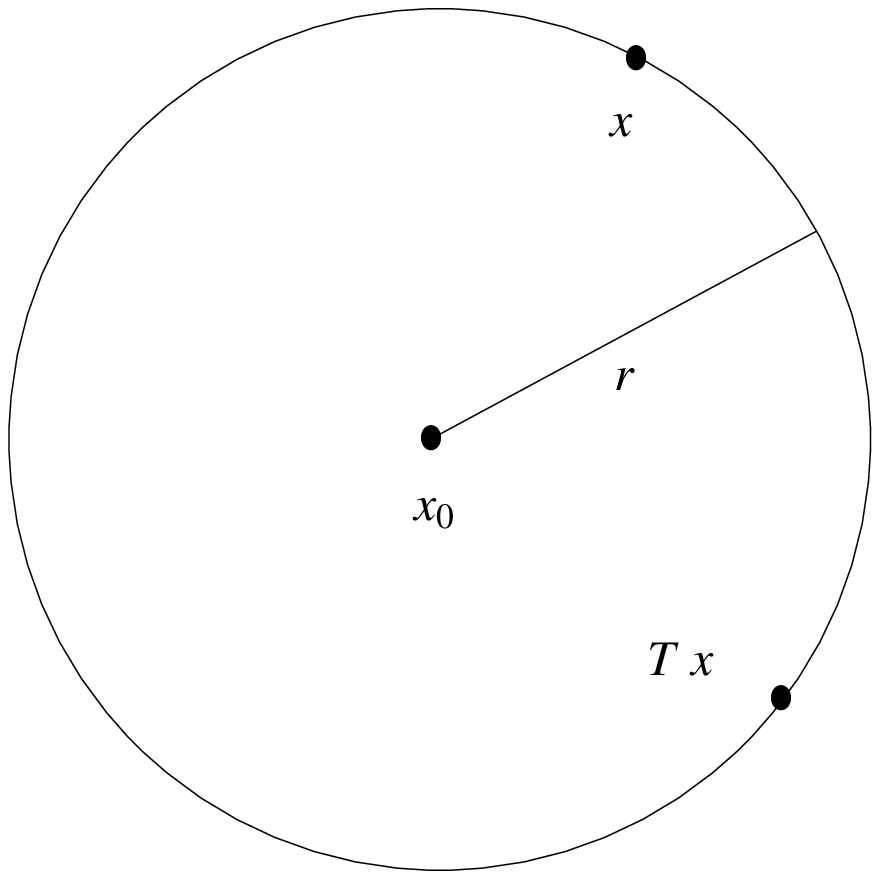}
  \caption{\Small  The condition $(C1)^{*} \cap (C2)^{*}$.}
  \label{fig:1B}
\end{subfigure}
\caption{\Small The geometric presentation of the conditions $(C1)^{\ast }$
and $(C2)^{\ast }$.}
\label{fig:2}
\end{figure}
\end{remark}

Now we give some fixed-circle examples.

\begin{example}
\label{exm8} Let $(X,d)$ be a metric space and $\alpha $ be a constant such
that%
\begin{equation*}
d(\alpha ,x_{0})<r\text{.}
\end{equation*}%
Let us consider a circle $C_{x_{0},r}$ and define the self-mapping $%
T:X\rightarrow X$ as%
\begin{equation*}
Tx=\left\{
\begin{array}{ccc}
x & \text{;} & x\in C_{x_{0},r} \\
\alpha & \text{;} & \text{otherwise}%
\end{array}%
\right. \text{,}
\end{equation*}%
for all $x\in X$. Then it can be easily checked that the conditions $%
(C1)^{\ast }$ and $(C2)^{\ast }$ are satisfied. Clearly $C_{x_{0},r}$ is a
fixed circle of the self-mapping $T$.
\end{example}

\begin{example}
\label{exm9} Let $(%
\mathbb{R}
,d)$ be the usual metric space and $C_{0,1}$ be the unit circle on $%
\mathbb{R}
$. Let us define the self-mapping $T:%
\mathbb{R}
\rightarrow
\mathbb{R}
$ as%
\begin{equation*}
Tx=\left\{
\begin{array}{ccc}
\dfrac{1}{x} & \text{;} & x\in C_{0,1} \\
5 & \text{;} & \text{otherwise}%
\end{array}%
\right. \text{,}
\end{equation*}%
for all $x\in
\mathbb{R}
$. Then the self-mapping $T$ satisfies the conditions $(C1)^{\ast }$ and $%
(C2)^{\ast }$. Hence $C_{0,1}$ is the fixed circle of $T$. Notice that the
fixed circle $C_{0,1}$ is not unique. $C_{3,2}$ and $C_{2,3}$ are also fixed
circles of $T$. It can be easily verified that $T$ satisfies the conditions $%
(C1)^{\ast }$ and $(C2)^{\ast }$ for the circles $C_{3,2}$ and $C_{2,3}$.
\end{example}

In the following example we give an example of a self-mapping which
satisfies the condition $(C2)^{\ast }$ and does not satisfy the condition $%
(C1)^{\ast }$.

\begin{example}
\label{exm10} Let $(X,d)$ be any metric space and $C_{x_{0},r}$ be any
circle on $X$. Let $\alpha $ be chosen such that $d(\alpha ,x_{0})=\rho <r$
and consider the self-mapping $T:X\rightarrow X$ defined by%
\begin{equation*}
Tx=\alpha \text{,}
\end{equation*}%
for all $x\in X$. Then the self-mapping $T$ satisfies the condition $%
(C2)^{\ast }$ but does not satisfy the condition $(C1)^{\ast }$. Clearly $T$
does not fix the circle $C_{x_{0},r}$.
\end{example}

In the following example we give an example of a self-mapping which
satisfies the condition $(C1)^{\ast }$ and does not satisfy the condition $%
(C2)^{\ast }$.

\begin{example}
\label{exm11} Let $(%
\mathbb{R}
,d)$ be the usual metric space and $C_{0,1}$ be the unit circle on $%
\mathbb{R}
$. Let us define the self-mapping $T:%
\mathbb{R}
\rightarrow
\mathbb{R}
$ as%
\begin{equation*}
Tx=\left\{
\begin{array}{ccc}
-5 & \text{;} & x=-1 \\
5 & \text{;} & x=1 \\
10 & \text{;} & \text{otherwise}%
\end{array}%
\right. \text{,}
\end{equation*}%
for all $x\in
\mathbb{R}
$. Then the self-mapping $T$ satisfies the condition $(C1)^{\ast }$ but does
not satisfy the condition $(C2)^{\ast }$. Clearly $T$ does not fix the
circle $C_{0,1}$ $($or any circle$)$.
\end{example}

Using the inequality (\ref{caristi_eqn}), we give another existence
fixed-circle theorem on a metric space.

\begin{theorem}
\label{thm7} Let $(X,d)$ be a metric space and $C_{x_{0},r}$ be any circle
on $X$. Let the mapping $\varphi $ be defined as $($\ref{phi function}$)$.
If there exists a self-mapping $T:X\rightarrow X$ satisfying

$(C1)^{\ast \ast }$ $d(x,Tx)\leq \varphi (x)-\varphi (Tx)$,\newline
and

$(C2)^{\ast \ast }$ $hd(x,Tx)+d(Tx,x_{0})\geq r$,\newline
for each $x\in C_{x_{0},r}$ and some $h\in \lbrack 0,1)$, then $C_{x_{0},r}$
is a fixed circle of $T$.
\end{theorem}

\begin{proof}
We consider the mapping $\varphi $ defined in $($\ref{phi function}$)$.
Assume that $x\in C_{x_{0},r}$ and $Tx\neq x$. Then using the conditions $%
(C1)^{\ast \ast }$ and $(C2)^{\ast \ast }$ we obtain%
\begin{eqnarray*}
d(x,Tx) &\leq &\varphi (x)-\varphi (Tx)=d(x,x_{0})-d(Tx,x_{0}) \\
&=&r-d(Tx,x_{0}) \\
&\leq &hd(x,Tx)+d(Tx,x_{0})-d(Tx,x_{0}) \\
&=&hd(x,Tx)\text{,}
\end{eqnarray*}%
which is a contradiction with our assumption since $h\in \lbrack 0,1)$.
Therefore we get $Tx=x$ and $C_{x_{0},r}$ is a fixed circle of $T$.
\end{proof}

\begin{remark}
\label{rem5} Notice that the condition $(C1)^{\ast \ast }$ guarantees that $%
Tx$ is not in the exterior of the circle $C_{x_{0},r}$ for each $x\in
C_{x_{0},r}$. The condition $(C2)^{\ast \ast }$ guarantees that $Tx$ should
be lies on or exterior or interior of the circle $C_{x_{0},r}$.
Consequently, $Tx$ should be lies on or interior of the circle $C_{x_{0},r}$
$($see Figure \ref{fig:3} for the geometric interpretation of the conditions
$(C1)^{\ast \ast }$ and $(C2)^{\ast \ast })$.

\begin{figure}[t]
\centering
\begin{subfigure}{.5\textwidth}
  \centering
  \includegraphics[width=.4\linewidth]{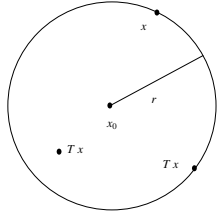}
  \caption{\Small The condition $(C1)^{**}$.}
  \label{fig:1A}
\end{subfigure}
\begin{subfigure}{.5\textwidth}
  \centering
  \includegraphics[width=.4\linewidth]{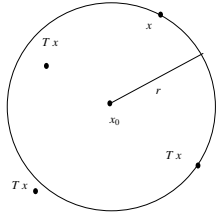}
  \caption{\Small  The condition $(C2)^{**}$.}
  \label{fig:1B}
\end{subfigure}
\begin{subfigure}{.5\textwidth}
  \centering
  \includegraphics[width=.4\linewidth]{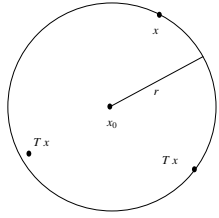}
  \caption{\Small  The condition $(C1)^{**} \cap (C2)^{**}$.}
  \label{fig:1B}
\end{subfigure}
\caption{\Small The geometric presentation of the conditions $(C1)^{\ast
\ast }$ and $(C2)^{\ast \ast }$.}
\label{fig:3}
\end{figure}
\end{remark}

\begin{example}
\label{exm13} Let $X=%
\mathbb{R}
$ and the mapping $d:X^{2}\rightarrow \lbrack 0,\infty )$ be defined as%
\begin{equation*}
d(x,y)=\left\vert e^{x}-e^{y}\right\vert \text{,}
\end{equation*}%
for all $x\in
\mathbb{R}
$. Then $(%
\mathbb{R}
,d)$ be a metric space. Let us consider the circle $C_{0,1}=\{\ln 2\}$ and
define the self-mapping $T:%
\mathbb{R}
\rightarrow
\mathbb{R}
$ as%
\begin{equation*}
Tx=\left\{
\begin{array}{ccc}
\ln 2 & ; & x\in C_{0,1} \\
1 & ; & \text{otherwise}%
\end{array}%
\right. \text{,}
\end{equation*}%
for all $x\in
\mathbb{R}
$. Then it can be easily checked that the conditions $(C1)^{\ast \ast }$ and
$(C2)^{\ast \ast }$ are satisfied. Hence the unit circle $C_{0,1}$ is a
fixed circle of $T$.
\end{example}

In the following example we give an example of a self-mapping which
satisfies the condition $(C1)^{\ast \ast }$ and does not satisfy the
condition $(C2)^{\ast \ast }$.

\begin{example}
\label{exm14} Let $(%
\mathbb{R}
,d)$ be the usual metric space. Let us consider the circle $C_{2,4}=\{-2,6\}$
and define the self-mapping $T:%
\mathbb{R}
\rightarrow
\mathbb{R}
$ as%
\begin{equation*}
Tx=\left\{
\begin{array}{ccc}
2 & \text{;} & x\in C_{2,4} \\
6 & \text{;} & \text{otherwise}%
\end{array}%
\right. \text{,}
\end{equation*}%
for all $x\in
\mathbb{R}
$. Then the self-mapping $T$ satisfies the condition $(C1)^{\ast \ast }$ but
does not satisfy the condition $(C2)^{\ast \ast }$. Clearly $T$ does not fix
the circle $C_{2,4}$ $($or any circle$)$.
\end{example}

In the following example we give an example of a self-mapping which
satisfies the condition $(C2)^{\ast \ast }$ and does not satisfy the
condition $(C1)^{\ast \ast }$.

\begin{example}
\label{exm15} Let $(%
\mathbb{R}
,d)$ be the usual metric space. Let us consider the circle $C_{0,2}$ and
define the self-mapping $T:%
\mathbb{R}
\rightarrow
\mathbb{R}
$ as%
\begin{equation*}
Tx=2\text{,}
\end{equation*}%
for all $x\in
\mathbb{R}
$. Then the self-mapping $T$ satisfies the condition $(C2)^{\ast \ast }$ but
does not satisfy the condition $(C1)^{\ast \ast }$. Clearly $T$ does not fix
the circle $C_{0,2}$ $($or any circle$)$.
\end{example}

\begin{example}
\label{exm12} Let $X=%
\mathbb{R}
$ and the mapping $d:X^{2}\rightarrow \lbrack 0,\infty )$ be defined as%
\begin{equation*}
d(x,y)=\left\{
\begin{array}{ccc}
0 & ; & x=y \\
\left\vert x\right\vert +\left\vert y\right\vert & ; & x\neq y%
\end{array}%
\right. \text{,}
\end{equation*}%
for all $x\in
\mathbb{R}
$. Then $(%
\mathbb{R}
,d)$ be a metric space. Let us define the self-mapping $T:%
\mathbb{R}
\rightarrow
\mathbb{R}
$ as%
\begin{equation*}
Tx=\left\{
\begin{array}{ccc}
\dfrac{1}{2} & ; & x\in \{-1,1\} \\
0 & ; & \text{otherwise}%
\end{array}%
\right. \text{,}
\end{equation*}%
for all $x\in
\mathbb{R}
$. Then the self-mapping $T$ does not satisfy the condition $(C1)^{\ast }$
but satisfies the condition $(C2)^{\ast }$ for the circle $C_{1,2}$. Hence $%
T $ does not fix the circle $C_{1,2}$. On the other hand it can be easily
checked that $T$ satisfies both of the conditions $(C1)^{\ast }$ and $%
(C2)^{\ast }$ for the circle $C_{1,1}$ and so fixes $C_{1,1}$. Actually
notice that $T$ fixes all of the circles centered at $x_{0}=a\in
\mathbb{R}
^{+}$ with radius $a$.
\end{example}

Let $I_{X}:X\rightarrow X$ be the identity map defined as $I_{X}(x)=x$ for
all $x\in X.$ Notice that the identity map satisfies the conditions $(C1)$
and $(C2)$ (resp. $(C1)^{\ast }$ and $(C2)^{\ast }$, $(C1)^{\ast \ast }$ and
$(C2)^{\ast \ast }$) in Theorem \ref{thm1} (resp. Theorem \ref{thm5} and
Theorem \ref{thm7}). Now we investigate a condition which excludes $I_{X}$
in Theorem \ref{thm1}, Theorem \ref{thm5} and Theorem \ref{thm7}. We give
the following theorem.

\begin{theorem}
\label{thm9} Let $(X,d)$ be a metric space and $C_{x_{0},r}$ be any circle
on $X$. Let the mapping $\varphi $ be defined as $(\ref{phi function})$. If
a self-mapping $T:X\rightarrow X$ satisfies the condition%
\begin{equation*}
(I_{d})\text{ \ \ \ \ \ \ }d(x,Tx)\leq \dfrac{\varphi (x)-\varphi (Tx)}{h}%
\text{,}
\end{equation*}%
for all $x\in X$ and some $h>1$ then $T=I_{X}$ and $C_{x_{0},r}$ is a fixed
circle of $T$.
\end{theorem}

\begin{proof}
Let $x\in X$ and $Tx\neq x$. Then using the inequality $(I_{d})$ and the
triangle inequality we get%
\begin{eqnarray*}
hd(x,Tx) &\leq &\varphi (x)-\varphi (Tx) \\
&=&d(x,x_{0})-d(Tx,x_{0}) \\
&\leq &d(x,Tx)+d(Tx,x_{0})-d(Tx,x_{0}) \\
&=&d(x,Tx)
\end{eqnarray*}%
and so%
\begin{equation*}
(h-1)d(x,Tx)\leq 0\text{,}
\end{equation*}%
which is a contradiction since $h>1$. Hence we obtain $Tx=x$ and $T=I_{X}$.
Consequently, $C_{x_{0},r}$ is a fixed circle of $T$.
\end{proof}

Notice that the converse statement of this theorem is also true. Hence if a
self-mapping $T$ in Theorem \ref{thm1} (resp. Theorem \ref{thm5} and Theorem %
\ref{thm7}) does not satisfy the condition $(I_{d})$ given in Theorem \ref%
{thm9} then $T$ can not be the identity map.

Considering the above examples we see that our existence theorems are
depending on the given circle (and so the metric on $X$). Also fixed circle
should not to be unique as seen in Example \ref{exm9}. Therefore it is
necessary and important to determine some uniqueness theorems for fixed
circles.

\section{\textbf{Some uniqueness theorems}}

\label{sec:2} In this section we investigate the uniqueness of the fixed
circles in theorems obtained in Section \ref{sec:1}. Notice that the fixed
circle $C_{x_{0},r}$ is not necessarily unique in Theorem \ref{thm1} (resp.
Theorem \ref{thm5} and Theorem \ref{thm7}). We can give the following result.

\begin{proposition}
\label{prop1} Let $(X,d)$ be a metric space. For any given circles $%
C_{x_{0},r}$ and $C_{x_{1},\rho }$, there exists at least one self-mapping $%
T $ of $X$ such that $T$ fixes the circles $C_{x_{0},r}$ and $C_{x_{1},\rho
} $.
\end{proposition}

\begin{proof}
Let $C_{x_{0},r}$ and $C_{x_{1},\rho }$ be any circles on $X$. Let us define
the self-mapping $T:X\rightarrow X$ as%
\begin{equation}
Tx=\left\{
\begin{array}{ccc}
x & \text{;} & x\in C_{x_{0},r}\cup C_{x_{1},\rho } \\
\alpha & \text{;} & \text{otherwise}%
\end{array}%
\right. \text{,}  \label{definition of T}
\end{equation}%
for all $x\in X$, where $\alpha $ is a constant satisfying $d(\alpha
,x_{0})\neq r$ and $d(\alpha ,x_{1})\neq \rho $. Let us define the mappings $%
\varphi _{1},\varphi _{2}:X\rightarrow \lbrack 0,\infty )$ as
\begin{equation*}
\varphi _{1}(x)=d(x,x_{0})
\end{equation*}%
and
\begin{equation*}
\varphi _{2}(x)=d(x,x_{1})\text{,}
\end{equation*}%
for all $x\in X$. Then it can be easily checked that the conditions $(C1)$
and $(C2)$ are satisfied by $T$ for the circles $C_{x_{0},r}$ and $%
C_{x_{1},\rho }$ with the mappings $\varphi _{1}(x)$ and $\varphi _{2}(x)$,
respectively. Clearly $C_{x_{0},r}$ and $C_{x_{1},\rho }$ are the fixed
circles of $T$ by Theorem \ref{thm1}.
\end{proof}

Notice that the circles $C_{x_{0},r}$ and $C_{x_{1},\rho }$ do not have to
be disjoint (see Example \ref{exm9}).

\begin{remark}
\label{rem1} Let $(X,d)$ be a metric space and $C_{x_{0},r}$, $C_{x_{1},\rho
}$ be two circles on $X$. If we consider the self-mapping $T$ defined in $($%
\ref{definition of T}$)$, then the conditions $(C1)^{\ast }$ and $(C2)^{\ast
}$ are satisfied by $T$ for the circles $C_{x_{0},r}$ and $C_{x_{1},\rho }$
with the mappings $\varphi _{1}(x)$ and $\varphi _{2}(x)$, respectively.
Clearly $C_{x_{0},r}$ and $C_{x_{1},\rho }$ are the fixed circles of $T$ by
Theorem \ref{thm5}. Similarly, the self-mapping $T$ in $($\ref{definition of
T}$)$ satisfies the conditions $(C1)^{\ast \ast }$ and $(C2)^{\ast \ast }$
for the circles $C_{x_{0},r}$ and $C_{x_{1},\rho }$ with the mappings $%
\varphi _{1}(x)$ and $\varphi _{2}(x)$, respectively.
\end{remark}

\begin{corollary}
\label{cor1} Let $(X,d)$ be a metric space. For any given circles $%
C_{x_{1},r_{1}}$,$\cdots $, $C_{x_{n},r_{n}}$, there exists at least one
self-mapping $T$ of $X$ such that $T$ fixes the circles $C_{x_{1},r_{1}}$,$%
\cdots $, $C_{x_{n},r_{n}}$.
\end{corollary}

\begin{example}
\label{exm7} Let $(X,d)$ be a metric space and $C_{x_{1},r_{1}}$,$\cdots $, $%
C_{x_{n},r_{n}}$ be any circles on $X$. Let $\alpha $ be a constant such that%
\begin{equation*}
d(\alpha ,x_{i})\neq r_{i}\text{ }(1\leq i\leq n)\text{.}
\end{equation*}%
Let us define the self-mapping $T:X\rightarrow X$ by%
\begin{equation*}
Tx=\left\{
\begin{array}{ccc}
x & \text{;} & x\in \bigcup\limits_{i=1}^{n}C_{x_{i},r_{i}} \\
\alpha & \text{;} & \text{otherwise}%
\end{array}%
\right. \text{,}
\end{equation*}%
for all $x\in X$ and the mappings $\varphi _{i}:X\rightarrow \lbrack
0,\infty )$ as
\begin{equation*}
\varphi _{i}(x)=d(x,x_{i})\text{ }(1\leq i\leq n).
\end{equation*}%
Then it can be easily checked that the conditions $(C1)$ and $(C2)$ are
satisfied by $T$ for the circles $C_{x_{1},r_{1}}$,$\cdots $, $%
C_{x_{n},r_{n}}$, respectively. Consequently, $C_{x_{1},r_{1}}$,$\cdots $, $%
C_{x_{n},r_{n}}$ are fixed circles of $T$ by Theorem \ref{thm1}. Notice that
these circles do not have to be disjoint.
\end{example}

Therefore it is important to investigate the uniqueness of the fixed
circles. Now we determine the uniqueness conditions for the fixed circles in
Theorem \ref{thm1}.

\begin{theorem}
\label{thm2} Let $(X,d)$ be a metric space and $C_{x_{0},r}$ be any circle
on $X$. Let $T:X\rightarrow X$ be a self-mapping satisfying the conditions $%
(C1)$ and $(C2)$ given in Theorem \ref{thm1}. If the contraction condition%
\begin{equation}
(C3)\text{ \ \ \ \ \ }d(Tx,Ty)\leq hd(x,y)\text{,}  \label{C3}
\end{equation}%
is satisfied for all $x\in C_{x_{0},r}$, $y\in X\setminus C_{x_{0},r}$ and
some $h\in \lbrack 0,1)$ by $T$, then $C_{x_{0},r}$ is the unique fixed
circle of $T$.
\end{theorem}

\begin{proof}
Assume that there exist two fixed circles $C_{x_{0},r}$ and $C_{x_{1},\rho }$
of the self-mapping $T$, that is, $T$ satisfy the conditions $(C1)$ and $%
(C2) $ for each circles $C_{x_{0},r}$ and $C_{x_{1},\rho }$. Let $u\in
C_{x_{0},r} $ and $v\in C_{x_{1},\rho }$ be arbitrary points. We show that $%
d(u,v)=0$ and hence $u=v$. Using the condition $(C3)$ we have%
\begin{equation*}
d(u,v)=d(Tu,Tv)\leq hd(u,v)\text{,}
\end{equation*}%
which is a contradiction since $h\in \lbrack 0,1)$. Consequently, $%
C_{x_{0},r}$ is the unique fixed circle of $T$.
\end{proof}

Notice that the self-mapping $T$ given in the proof of Proposition \ref%
{prop1} does not satisfy the contraction condition $(C3)$.

We give a uniqueness condition for the fixed circles in Theorem \ref{thm5}.

\begin{theorem}
\label{thm6} Let $(X,d)$ be a metric space and $C_{x_{0},r}$ be any circle
on $X$. Let $T:X\rightarrow X$ be a self-mapping satisfying the conditions $%
(C1)^{\ast }$ and $(C2)^{\ast }$ given in Theorem \ref{thm5}. If the
contraction condition $(C3)$ defined in $($\ref{C3}$)$ is satisfied for all $%
x\in C_{x_{0},r}$, $y\in X\setminus C_{x_{0},r}$ and some $h\in \lbrack 0,1)$
by $T$ then $C_{x_{0},r}$ is the unique fixed circle of $T$.
\end{theorem}

\begin{proof}
It can be easily seen by the same arguments used in the proof of Theorem \ref%
{thm2}.
\end{proof}

Finally we give a uniqueness condition for the fixed circles in Theorem \ref%
{thm7}.

\begin{theorem}
\label{thm8} Let $(X,d)$ be a metric space and $C_{x_{0},r}$ be any circle
on $X$. Let $T:X\rightarrow X$ be a self-mapping satisfying the conditions $%
(C1)^{\ast \ast }$ and $(C2)^{\ast \ast }$ given in Theorem \ref{thm7}. If
the contraction condition%
\begin{equation*}
(C3)^{\ast \ast }\text{ \ \ }d(Tx,Ty)<\max
\{d(x,y),d(x,Tx),d(y,Ty),d(x,Ty),d(y,Tx)\}\text{,}
\end{equation*}%
is satisfied for all $x\in C_{x_{0},r}$, $y\in X\setminus C_{x_{0},r}$ by $T$%
, then $C_{x_{0},r}$ is the unique fixed circle of $T$.
\end{theorem}

\begin{proof}
Suppose that there exist two fixed circles $C_{x_{0},r}$ and $C_{x_{1},\rho
} $ of the self-mapping $T$, that is, $T$ satisfy the conditions $(C1)^{\ast
\ast }$ and $(C2)^{\ast \ast }$ for each circles $C_{x_{0},r}$ and $%
C_{x_{1},\rho }$. Let $u\in C_{x_{0},r}$, $v\in C_{x_{1},\rho }$ and $u\neq
v $ be arbitrary points. We show that $d(u,v)=0$ and hence $u=v$. Using the
condition $(C3)^{\ast \ast }$ we have%
\begin{eqnarray*}
d(u,v) &=&d(Tu,Tv)<\max \{d(u,v),d(u,Tu),d(v,Tv),d(u,Tv),d(v,Tu)\} \\
&=&d(u,v)\text{,}
\end{eqnarray*}%
which is a contradiction. Consequently, it should be $u=v$ for all $u\in
C_{x_{0},r}$, $v\in C_{x_{1},\rho }$ and so $C_{x_{0},r}$ is the unique
fixed circle of $T$.
\end{proof}

Notice that the uniqueness of the fixed circle in Theorem \ref{thm1} and
Theorem \ref{thm5} can be also obtained using the contraction condition $%
(C3)^{\ast \ast }$. Similarly, the uniqueness of the fixed circle in Theorem %
\ref{thm7} can be also obtained using the contraction condition $(C3)$. More
generally it is possible to use appropriate contractive conditions for the
uniqueness of the fixed-circle theorems obtained in Section \ref{sec:1}.


\begin{thebibliography}{9}
\bibitem{Caristi-1976} Caristi, J. Fixed point theorems for mappings
satisfying inwardness conditions. Trans. Amer. Math. Soc. 1976; 215: 241-251.

\bibitem{Ciesielski-2007} Ciesielski, K. On Stefan Banach and some of his
results. Banach J. Math. Anal. 2007; 1: 1-10.

\bibitem{Jones-Singerman} Jones, G. A. and Singerman, D. Complex functions
an algebric and geometric viewpoint. Cambridge University Press, New York,
1987.

\bibitem{Mandic-2000} Mandic, D. P. The use of M\"{o}bius transformations in
neural networks and signal processing. Neural Networks for Signal Processing
- Proceedings of the IEEE Workshop, 2000; 1: 185-194.

\bibitem{Ozdemir-2011} \"{O}zdemir, N., \.{I}skender, B. B. and \"{O}zg\"{u}%
r, N. Y. Complex valued neural network with M\"{o}bius activation function.
Commun. Nonlinear Sci. Numer. Simul. 2011; 16: 4698-4703.

\bibitem{Rhoades} Rhoades, B. E. A comparison of various definitions of
contractive mappings. Trans. Amer. Math. Soc. 1977; 226: 257-290.
\end{thebibliography}
\end{document}